\documentclass{amsart}
\usepackage[left=1.21in,right=1.21in]{geometry}
\newtheorem{theorem}{Theorem}[section]
\newtheorem{lemma}[theorem]{Lemma}
\newtheorem{proposition}[theorem]{Proposition}
\theoremstyle{definition}

\newtheorem{typecase}[theorem]{Case}
\theoremstyle{remark}
\newtheorem{remark}[theorem]{Remark}
\numberwithin{equation}{section}
\usepackage{booktabs}
\usepackage{array}
\newcommand{\seq}[1]{\par\smallskip\noindent{\raggedright #1\par}\smallskip}
\usepackage{url}
\usepackage{hyperref}
\hypersetup{%
  colorlinks=true,
  linkcolor=blue,
  citecolor=blue,
  urlcolor=blue,
  pdftitle={On Permutations Whose Adjacent-Sum Reciprocals Sum to One},
}
\begin{document}
\title{On Permutations Whose Adjacent-Sum Reciprocals Sum to One}
\author{Jinggao Ai}
\address{School of Mathematics and Statistics, Wuhan University, Wuhan 430072, China}
\email{Jules142857@gmail.com}
\author{Dingsheng Hu}
\address{School of Mathematics and Statistics, Wuhan University, Wuhan 430072, China}
\email{hudingsheng618@gmail.com}
\author{Junmin Yan}
\address{School of Mathematics and Statistics, Wuhan University, Wuhan 430072, China}
\email{glorysensitives@gmail.com}
\begin{abstract}
In this paper, we give a constructive proof that for any integer $n>7$ (resp. $n>6$), the numbers $1,2,\ldots,n$ can be arranged in a circle (resp. in a line) such that the sum of the reciprocals of all adjacent sums equals $1$. The above conjecture was originally posed by Zhi-Wei Sun.
Partitioning the integers according to $n\bmod 12$ into nine residue classes, we prove a recursive proposition and inductively construct new examples. Main lemmas and propositions in introduction are proved in Section~\ref{sec:proof}. The rest of the cases are handled by explicit constructions in appendix.
\end{abstract}
\maketitle
\section{Introduction}
\subsection*{Background}
Permutations are important objects of study in mathematics (see, e.g.,
\cite{Sedgewick1977}). Let $\mathfrak{S}_n$ be the symmetric group of permutations of $\{1,2,\ldots,n\}$. Let $\tau\in \mathfrak{S}_n$ and we may represent it by the sequence
$(\tau(1),\ldots,\tau(n))$. In \cite{Sun2018}, Z.-W. Sun posed several
conjectures about permutations. In this paper we resolve one of them [Conj. 4.7(i)]
(namely Theorem~\ref{thm:main}).
For related work on decompositions into distinct parts with a prescribed sum of reciprocals, see~\cite{vanDoorn2025}.
\begin{theorem}\label{thm:main}
{\rm(i)}For any integer $n>6$, there exists a permutation $\tau\in \mathfrak{S}_n$ such that
\begin{equation}\label{eq:linear}
\sum_{i=1}^{n-1}\frac{1}{\tau(i)+\tau(i+1)}=1.
\end{equation}
{\rm(ii)}For any integer $n>7$, there exists a permutation $\tau\in \mathfrak{S}_n$
such that
\begin{equation}\label{eq:cyclic}
\sum_{i=1}^{n-1}\frac{1}{\tau(i)+\tau(i+1)}+ \frac{1}{\tau(n)+\tau(1)}=1.
\end{equation}
\end{theorem}
\medskip
Among the many permutation conjectures of Z.-W. Sun~\cite{Sun2018,Sun2019add,Sun2021book},
G.-N. Han~\cite[Theorem~1]{Han2020} proved three: for all sufficiently large
$n$, there exist permutations $\tau,\pi,\sigma\in\mathfrak S_n$ such that
\[
\sum_{k=1}^{n-1}\frac{1}{\tau(k)-\tau(k+1)}=0,\qquad
\sum_{k=1}^{n-1}\frac{1}{\pi(k)-\pi(k+1)}+\frac{1}{\pi(n)-\pi(1)}=0,
\qquad
\sum_{k=1}^{n-1}\frac{1}{\sigma(k)\sigma(k+1)}=1.
\]
\medskip
Sun also conjectured (see ~\cite{Sunsun}) that for any integer $n>5$ there is a permutation $\tau\in \mathfrak S_n$ satisfying
\[
\sum_{k=1}^n \frac{1}{k+\tau(k)}=1,
\]
and this was later confirmed by the user Zhao Shen on MathOverflow via a
clever induction argument.\par
\medskip
A classic conjecture regarding adjacent sums of permutations was posed by
Filz(~\cite{Filz1982}, pp160-162), and later recorded in Guy's book~\cite{Guy2004}, which states that there exists  $\tau\in\mathfrak S_n$ for each positive even number $n$ such that the $n$
adjacent sums
\[
\tau(1)+\tau(2),\ \tau(2)+\tau(3),\ \ldots,\ \tau(n-1)+\tau(n),\ \tau(n)+\tau(1)
\]
are all prime. This conjecture remains open (see~\cite{ChenFuGuo2021,ShangLiZhang2025,She2026}).
In fact, Y.-F. She~\cite{She2026} proved that the assertion holds for all
sufficiently large even $n$.\par
\medskip
\subsection*{Sketch of the proof}
We now sketch the proof of sufficiently large \(n\). For convenience, we denote the left-hand side of \eqref{eq:linear} by $L(\tau)$, and the left-hand
side of \eqref{eq:cyclic} by $C(\tau)$. When computing $L(\tau)$, we call $\tau$ a
\emph{linear arrangement} of length $n$; when computing $C(\tau)$, we call $\tau$ a
\emph{cyclic arrangement} of length $n$, in which case we adopt the convention
$\tau(n+1)=\tau(1)$.\par
For a given \(n\), a cyclic arrangement
\(\tau_0\in\mathfrak{S}_n\) satisfying \eqref{eq:cyclic} and some conditions (see Proposition~\ref{lem:splice}) results in a linear
arrangement \(\tau\in\mathfrak{S}_n\) satisfying \eqref{eq:linear}. Moreover, Theorem~\ref{thm:main}(i) in the case
\(n=7\) constructed in Appendix~\ref{app:exceptions}, and
we thus focus on proving Theorem~\ref{thm:main}(ii). \par
The goal of our inductive construction starts from a cyclic arrangement of length $n$ satisfying some conditions (see Proposition~\ref{lem:recip} and ~\ref{lem:main-recurse}), and inserts $2T$ new elements into it to obtain a cyclic arrangement of length $n+2T$ still satisfying the same conditions; the step can then be repeated indefinitely (see
Proposition~\ref{lem:main-recurse}).\par
\medskip
For finite sequences $X=(x_1,\ldots,x_k)$ and $Y=(y_1,\ldots,y_m)$, write $X\circ Y=(x_1,\ldots,x_k,y_1,\ldots,y_m)$ (or briefly $(X,Y):=X\circ Y$), and $X^{\mathrm{rev}}=(x_k,\ldots,x_1)$ for the reversal of $X$.\par
We also use the following terminology. Given a cyclic arrangement $\tau\in\mathfrak S_n$, each unordered pair $\{\tau(k),\tau(k+1)\}$ of adjacent elements is called an \emph{undirected edge}. When the direction $\tau(1)\to\tau(2)\to\cdots\to\tau(n)\to\tau(1)$ is fixed, the ordered pair $(\tau(k),\tau(k+1))$ is called a \emph{directed edge}.
\begin{proposition}\label{lem:splice}
Let
\[
\tau_c=(a,b)\circ X\circ(c,d)\circ Y\in \mathfrak{S}_n
\]
be a cyclic arrangement of length $n$, where $X,Y$ are subsequences (possibly empty),
and $a,b,c,d\in\{1,2,\ldots,n\}$ satisfy
\[
a+b=c+d=D,\qquad b+d=\frac{D}{2} .
\]
Then the linear arrangement of length $n$
\[
\tau_\ell=(a)\circ Y^{\mathrm{rev}}\circ(d,b)\circ X\circ(c)
\]
satisfies $L(\tau_\ell)=C(\tau_c)$. In this case, $(a,b)$
and $(c,d)$ are called \emph{splicing edges}.
\end{proposition}
\medskip
We divide the residue classes of the positive integer $n$ into nine
types: types $t=0,2,3$ each denote one class modulo $6$
($n\equiv t\pmod 6$), and types $t=5,11,1,7,4,10$ each denote one class
modulo $12$ ($n\equiv t\pmod{12}$). That is, the modulo-$6$ classes
$5,1,4$ are further subdivided modulo $12$.
\medskip
We next introduce the denominator notation used throughout the construction.
Let $\tau$ be a cyclic arrangement of length $n$. For $i=1,\ldots,n$ write $d_i=\tau(i)+\tau(i+1)$, called a \emph{denominator} of $\tau$, and let $D(\tau)=\{d_i\mid i=1,\ldots,n\}$ be the multiset of all denominators.
For each type, we call the denominators listed in the second
and third columns of Table~\ref{tab:denoms}
\emph{$\alpha$-denominators} and \emph{$\beta$-denominators}, respectively.
Their multiplicities are, respectively,
nonconstant affine functions of $n$ and constants. An edge whose endpoints sum to an $\alpha$-denominator (resp.\
$\beta$-denominator) is called an \emph{$\alpha$-edge} (resp.\
$\beta$-edge).
\begin{center}
\small
\begin{tabular}{c|l|l}
\toprule
Type & $\alpha$-denominator (multiplicity) & $\beta$-denominator (multiplicity) \\
\midrule
0 &
$n-2\,((n-6)/2)$,\ $n\,((n-12)/6)$,\ $n+3\,((n+3)/3)$
&
$3(n-2)/2\,(3)$,\ $n/2\,(1)$ \\
\bottomrule
\end{tabular}
\end{center}
\smallskip
The table above only shows the Type $0$ case. See Table~\ref{tab:denoms} for a complete table.
\smallskip
\begin{remark}
In our construction, every $\beta$-denominator is of the form $qD$, where $D$ is an $\alpha$-denominator and $q\in\{1/2,3/2,2\}$.
\end{remark}
\begin{proposition}\label{lem:recip}
According to the multiplicities in Table~\ref{tab:denoms}, the sum of reciprocals of
all denominators is identically $1$, and the total number of
denominators is exactly $n$.
\end{proposition}
\medskip
We now develop several tools for the main recursive proposition (see
Proposition~\ref{lem:main-recurse}). For each type, fix a sequence
$S=(s_0,\ldots,s_{L-1})$ and a parameter $T$
as in Table~\ref{tab:param}.
For each induction chain, fix a pair $(p,r)$
as in Table~\ref{tab:startparams}.
Thus $S,L,T$ depend only on the type, whereas
$p,r$ depend on the chosen induction chain. Denote that \(A\) is the first \(\alpha\)-denominator listed in Table~\ref{tab:denoms} for certain type. For the given $r$, define $f:\mathbb Z\to\mathbb Z$ by
\[
f(x)=x+s_{(x+r)\bmod L},
\]
here $s_{(x+r)\bmod L}$ is called the \emph{increment}. The parameters
$L$ and $T$ corresponding to the all nine types are given in Table~\ref{tab:param}.
\medskip
\begin{center}
\small
\begin{tabular}{c|c|c|c|c}
\toprule
Type & $A$ & $S$ & $L$ & $T$ \\
\midrule
0 & $n-2$ & $(2,5,5)$ & $3$ & $12$ \\
\bottomrule
\end{tabular}
\end{center}
Combined with the values of the $\alpha$-denominators in Table~\ref{tab:denoms}, one
sees that the values of $A+s_{(x+r)\bmod L}$ are precisely all possible
values of the $\alpha$-denominators other than $A$. For every pair
$(p,r)$ actually used in the appendix, we have
\begin{equation}\label{eq:growth}
x>p\implies f(x)>p.
\end{equation}
Since the minimum step is no less than $-3$, it suffices to check the
inequality for $p+1,p+2,p+3$ against the $(p,r)$ in the Appendix~\ref{app:wx}.
Moreover, since $T$ is a multiple of $L$, the same property holds for
$(p+T)$. The map
\[
x\bmod L\longmapsto f(x)\bmod L
\]
is a permutation. When each residue cycle is traversed once, its length and total increment all are as shown in Table~\ref{tab:cycles}. 
\begin{center}
\small
\begin{tabular}{c|l}
\toprule
Type & (cycle length, increment per cycle) \\
\midrule
0 & $(3,12)$ \\
\bottomrule
\end{tabular}
\end{center}
\medskip
Define the \emph{orbit}
\[
O(x)=\{\ldots,f^{-2}(x),f^{-1}(x),x,f(x),f^2(x),\ldots\}.
\]
The total increment of each residue cycle is positive and divides $T$; hence every
integer orbit tends to $+\infty$ forward and to $-\infty$ backward,
and $x$ and $x+T$ lie on the same orbit.
\medskip
We next introduce the crossing set, which records exactly where the orbits of $f$ cross the threshold $p$.
Define the \emph{crossing set}
\[
J=\{x\in\mathbb Z:x\le p<f(x)\}.
\]
For $x\in J$, set $z=f(x)$. Starting from $z$ and iterating $f$ until
first reaching $z+T$, denote the sequence before the endpoint by
\[
W_x=(z_1,z_2,\cdots,z_\ell),\qquad z_1=z,\quad
f(z_j)=z_{j+1},\quad f(z_\ell)=z+T.
\]
\begin{lemma}\label{lem:wx}
The sequences $W_x$ $(x\in J)$ are pairwise disjoint, and
\[
\bigcup_{x\in J}W_x=\{p+1,p+2,\cdots,p+T\}.
\]
\end{lemma}
\medskip
\subsection*{The main recursive proposition.}
Fix the type and \((p,r)\), and assume \(A>2p\), where \(A\) is the first \(\alpha\)-denominator listed in Table~\ref{tab:denoms} for this type. Define the region index
\begin{equation}\label{eq:rho}
\rho(x)=
\begin{cases}
0, & x\le p,\\
1, & p<x<A-p,\\
2, & x\ge A-p.
\end{cases}
\end{equation}
Assume the current cyclic arrangement satisfies the conditions of
Table~\ref{tab:denoms}, and also:\par
\medskip
\noindent\textbf{(H1)} For every $x\in J$, we have
\[
1\le A-x\le n,\qquad p<f(x)<A-p,
\]
and there exists an undirected edge
\begin{equation}\label{eq:expe}
\{A-x,f(x)\}.
\end{equation}
This is called the \emph{expansion edge} corresponding to $x$.\par
\medskip
\noindent\textbf{(H2)} For every non-expansion $\alpha$-edge, the sum of the
region indices of its two endpoints is $2$.\par
\medskip
\noindent\textbf{(H3)} If the denominator of a $\beta$-edge equals $qD$
for some $\alpha$-denominator $D$ and $q\in\{1/2,3/2,2\}$, then the sum
of the region indices of its two endpoints is $2q$.\par
\medskip
\noindent\textbf{(H4)} There exist two directed $\alpha$-edges
\[
a\to b,\qquad c\to d,
\]
not listed in \eqref{eq:expe}, whose four endpoints are pairwise distinct,
satisfying
\begin{equation}\label{eq:splicecond}
a+b=c+d=D=2(b+d),\qquad b\le p<d<A-p.
\end{equation}
In this case, $(a,b)$ and $(c,d)$ are called \emph{splicing edges}.\par
\medskip
We may now state the recursive proposition:
\begin{proposition}\label{lem:main-recurse}
If a cyclic arrangement of length $n$ satisfies the denominator
multiplicity conditions of Table~\ref{tab:denoms} and hypotheses (H1)--(H4), then a
new cyclic arrangement of length $n+2T$ can be constructed: the new
arrangement still satisfies the multiplicity pattern of Table~\ref{tab:denoms} and
all hypotheses (H1)--(H4), and can still be transformed into the
corresponding linear arrangement via Proposition~\ref{lem:splice}. Hence the above
expansion operation can be iterated indefinitely.
\end{proposition}
We have sketched the main ideas of the proof in this introduction. In Section~\ref{sec:proof}, we prove the propositions and lemmas stated above. In Section~\ref{sec:base}, we discuss the base cases and show that the induction chains together with the finite exceptions cover all $n>7$. This completes the proof of Theorem~\ref{thm:main}.
\section{Proofs}\label{sec:proof}
In this section we prove the propositions and lemmas stated above. We first show that a cyclic arrangement satisfying the conditions of Proposition~\ref{lem:splice} gives rise to a linear arrangement with the same reciprocal sum. We then check that, with the prescribed multiplicities, the sum of the reciprocals of all denominators equals $1$ (see Proposition~\ref{lem:recip})., and prove that the sequences $W_x$ defined in the introduction are pairwise disjoint and cover $(p,p+T]$ (see Lemma~\ref{lem:wx}). The full tables of denominator multiplicities, parameters, and cycle lengths for all nine types are showed in this section alongside the relevant proofs. Finally, we prove the main inductive step (see Proposition~\ref{lem:main-recurse}): inserting $2T$ new elements into an arrangement satisfying (H1)--(H4) produces a larger arrangement that again satisfies (H1)--(H4), so that the construction can be repeated indefinitely.
\subsection*{Proof of Proposition~\ref{lem:splice}}
The arrangement \(\tau_\ell\) remains a permutation of \(\{1, \dots, n\}\).
Compared with the original cycle, the two directed edges $a\to b$ and
$c\to d$ (both with denominator $D$) are removed, and a new undirected
edge $\{b,d\}$ (with denominator $D/2$) is added. Hence
\[
L(\tau_\ell)-C(\tau_c)=-\frac{1}{D}-\frac{1}{D}+\frac{1}{D/2}=0.
\]

\begin{table}[ht]
\centering
\caption{Denominators and their multiplicities}\label{tab:denoms}
\footnotesize
\setlength{\tabcolsep}{5pt}
\begin{tabular}{c|l|l}
\toprule
Type & $\alpha$-denominator (multiplicity) & $\beta$-denominator (multiplicity) \\
\midrule
0 &
\begin{tabular}[t]{@{}r@{\;}l@{}}
$n-2$   & $((n-6)/2)$  \\
$n$     & $((n-12)/6)$ \\
$n+3$   & $((n+3)/3)$
\end{tabular} &
\begin{tabular}[t]{@{}r@{\;}l@{}}
$3(n-2)/2$ & $(3)$ \\
$n/2$      & $(1)$
\end{tabular} \\
\addlinespace
2 &
\begin{tabular}[t]{@{}r@{\;}l@{}}
$n-4$   & $((n-8)/2)$  \\
$n+1$   & $((n+1)/3)$  \\
$n+10$  & $((n-2)/6)$
\end{tabular} &
\begin{tabular}[t]{@{}r@{\;}l@{}}
$3(n-4)/2$   & $(3)$ \\
$(n+10)/2$   & $(1)$
\end{tabular} \\
\addlinespace
5 &
\begin{tabular}[t]{@{}r@{\;}l@{}}
$n-4$   & $((n-5)/2)$  \\
$n-5$   & $((n-17)/6)$ \\
$n+3$   & $((n+3)/4)$  \\
$n+7$   & $((n+7)/12)$
\end{tabular} &
\begin{tabular}[t]{@{}r@{\;}l@{}}
$2(n-4)$      & $(1)$ \\
$3(n-5)/2$    & $(3)$
\end{tabular} \\
\addlinespace
11 &
\begin{tabular}[t]{@{}r@{\;}l@{}}
$n-4$   & $((n-5)/2)$   \\
$n-7$   & $((n-23)/12)$ \\
$n+1$   & $((n+1)/12)$  \\
$n+3$   & $((n+1)/3)$
\end{tabular} &
\begin{tabular}[t]{@{}r@{\;}l@{}}
$2(n-4)$       & $(1)$ \\
$3(n-7)/2$     & $(2)$ \\
$3(n+3)/2$     & $(1)$
\end{tabular} \\
\addlinespace
1 &
\begin{tabular}[t]{@{}r@{\;}l@{}}
$n+1$   & $((n-3)/2)$ \\
$n+5$   & $((n-1)/3)$ \\
$n-1$   & $((n-1)/6)$
\end{tabular} &
\begin{tabular}[t]{@{}r@{\;}l@{}}
$(n+1)/2$ & $(1)$ \\
$(n+5)/2$ & $(1)$
\end{tabular} \\
\addlinespace
7 &
\begin{tabular}[t]{@{}r@{\;}l@{}}
$n-2$   & $((n-3)/2)$ \\
$n-3$   & $((n-7)/6)$ \\
$n-1$   & $((n-1)/6)$ \\
$n+1$   & $((n-7)/6)$
\end{tabular} &
\begin{tabular}[t]{@{}r@{\;}l@{}}
$2(n-2)$      & $(1)$ \\
$3(n-3)/2$    & $(1)$ \\
$3(n+1)/2$    & $(2)$
\end{tabular} \\
\addlinespace
3 &
\begin{tabular}[t]{@{}r@{\;}l@{}}
$n-2$   & $((n-3)/2)$ \\
$n-3$   & $((n-3)/6)$ \\
$n-1$   & $((n-9)/6)$ \\
$n+1$   & $((n-3)/6)$
\end{tabular} &
\begin{tabular}[t]{@{}r@{\;}l@{}}
$2(n-2)$      & $(1)$ \\
$3(n-1)/2$    & $(2)$ \\
$3(n+1)/2$    & $(1)$
\end{tabular} \\
\addlinespace
4 &
\begin{tabular}[t]{@{}r@{\;}l@{}}
$n-2$   & $((n-6)/2)$  \\
$n-4$   & $((n-4)/12)$ \\
$n+2$   & $((n-10)/6)$ \\
$n+4$   & $((n+4)/4)$
\end{tabular} &
\begin{tabular}[t]{@{}r@{\;}l@{}}
$3(n-2)/2$  & $(3)$ \\
$(n+2)/2$   & $(1)$
\end{tabular} \\
\addlinespace
10 &
\begin{tabular}[t]{@{}r@{\;}l@{}}
$n$     & $((n-4)/2)$ \\
$n+2$   & $((n+2)/3)$ \\
$n+8$   & $((n-4)/6)$
\end{tabular} &
\begin{tabular}[t]{@{}r@{\;}l@{}}
$n/2$      & $(1)$ \\
$(n+8)/2$  & $(1)$
\end{tabular} \\
\bottomrule
\end{tabular}
\end{table}

\subsection*{Proof of Proposition~\ref{lem:recip}}
Take type $0$ as an example:
\[
\frac{(n-6)/2}{n-2}+\frac{(n-12)/6}{n}+\frac{(n+3)/3}{n+3}
+\frac{3}{3(n-2)/2}+\frac{1}{n/2}
=\frac12+\frac16+\frac13=1.
\]
After merging the $\beta$-denominators as multiples of the
corresponding $\alpha$-denominators, the other eight types reduce
respectively to
\[
\text{Type 2: }\tfrac12+\tfrac13+\tfrac16,\qquad
\text{Type 5: }\tfrac12+\tfrac16+\tfrac14+\tfrac1{12},
\]
\[
\text{Type 11: }\tfrac12+\tfrac1{12}+\tfrac1{12}+\tfrac13,\qquad
\text{Type 1: }\tfrac12+\tfrac13+\tfrac16,
\]
\[
\text{Type 7: }\tfrac12+\tfrac16+\tfrac16+\tfrac16,\qquad
\text{Type 3: }\tfrac12+\tfrac16+\tfrac16+\tfrac16,
\]
\[
\text{Type 4: }\tfrac12+\tfrac1{12}+\tfrac16+\tfrac14,\qquad
\text{Type 10: }\tfrac12+\tfrac13+\tfrac16.
\]
The total number of denominators being $n$ can be checked by direct
computation.
\subsection*{Proof of Lemma~\ref{lem:wx}}
By \eqref{eq:growth}, an orbit cannot return from $>p$ to $\le p$. Each orbit crosses from $\le p$ to $>p$ at a unique step $x\mapsto f(x)$ with $x\in J$. The same
holds for $p+T$; the unique crossing is
\[
(x+T)\longmapsto f(x)+T=z+T.
\]
Therefore, the portion of each orbit lying in the interval $(p,p+T]$
is exactly the corresponding sequence $W_x$. Different orbits are
disjoint, so all $W_x$ are pairwise disjoint; and since every integer
belongs to exactly one orbit, every integer in
$\{p+1,p+2,\ldots,p+T\}$ belongs to some $W_x$, i.e.\
$\bigcup_{x\in J}W_x=\{p+1,p+2,\ldots,p+T\}$.

\begin{table}[ht]
\centering
\caption{Parameter table}\label{tab:param}
\begin{tabular}{c|c|c|c|c}
\toprule
Type & $A$ & $S$ & $L$ & $T$ \\
\midrule
0  & $n-2$ & $(2,5,5)$                       & 3 & 12 \\
2  & $n-4$ & $(14,5,5)$                      & 3 & 24 \\
5  & $n-4$ & $(-1,7,-1,7,11,7)$               & 6 & 18 \\
11 & $n-4$ & $(-3,7,5,7,7,7)$                 & 6 & 36 \\
1  & $n+1$ & $(-2,4,4)$                      & 3 & 6  \\
7  & $n-2$ & $(-1,-1,1,1,3,3)$                & 6 & 6  \\
3  & $n-2$ & $(-1,-1,1,1,3,3)$                & 6 & 6  \\
4  & $n-2$ & $(-2,6,4,6,4,6)$                 & 6 & 6  \\
10 & $n$   & $(2,2,8)$                       & 3 & 12 \\
\bottomrule
\end{tabular}
\end{table}

\begin{table}[ht]
\centering
\caption{Cycle lengths and increments}\label{tab:cycles}
\begin{tabular}{c|l}
\toprule
Type & (cycle length, increment per cycle) \\
\midrule
0  & $(3,12)$ \\
2  & $(3,24)$ \\
5  & $(2,6),\ (2,6),\ (2,18)$ \\
11 & $(4,18),\ (2,12)$ \\
1  & $(3,6)$ \\
7  & $(6,6)$ \\
3  & $(6,6)$ \\
4  & $(3,6),\ (1,6),\ (1,6),\ (1,6)$ \\
10 & $(3,12)$ \\
\bottomrule
\end{tabular}
\end{table}

\subsection*{Proof of Proposition~\ref{lem:main-recurse}}
Set
\[
n'=n+2T,\qquad A'=A+2T.
\]
Define an embedding $E:\{1,2,\ldots,n\}\to\{1,2,\ldots,n'\}$ by
\begin{equation}\label{eq:E}
E(x)=x+T\rho(x)=
\begin{cases}
x, & x\le p,\\
x+T, & p<x<A-p,\\
x+2T, & x\ge A-p.
\end{cases}
\end{equation}
Under this map, the image of \{1,2,\ldots,n\}\ leaves exactly two intervals uncovered:
\begin{equation}\label{eq:gaps}
[p+1,p+T],\qquad [A+T-p,A+2T-p-1].
\end{equation}
Since $A>2p$, these two intervals are disjoint, and they correspond to
each other under the map $z\mapsto A'-z$. Hence both intervals have
length $T$ and together contain $2T$ integers. For $x\in J$, set $z=f(x)$, and let
$W_x=(z_1,\ldots,z_\ell)$. The old edge $(A-x,z)$ becomes
$(A'-x,z+T)$ under embedding E. Now we replace it by the path
\begin{equation}\label{eq:path}
(A'-x,z_1,A'-z_1,z_2,A'-z_2,\cdots,z_\ell,A'-z_\ell,z+T).
\end{equation}
If the original cycle traverses this edge in the reverse direction,
use the reverse order of \eqref{eq:path}.\par
We now check that the new cyclic arrangement still
satisfies the conditions of Table~\ref{tab:denoms}. By Lemma~\ref{lem:wx}, all $z_j$ exactly
fill $[p+1,p+T]$. All $A'-z_j$ exactly fill the second gap of \eqref{eq:path} . The
expansion edges for different $x$ are distinct, because their
region-$2$ endpoints $A-x$ are different. Hence, the new numbers are distinct and exactly fill the gap \eqref{eq:gaps}. Each operation replaces an edge of the cycle with a path , so the result is still a single cycle rather than multiple small cycles.
If an $\alpha$-edge has endpoint region indices summing to $2$, then under \eqref{eq:E} its denominator increases by $2T$, which is equivalent to replacing $n$ by $n'$ in the original denominator.
If a $\beta$-edge has a denominator $qD$, where $D$ is an
$\alpha$-denominator, then the sum of its endpoint region
indices is $2q$, so
\[
qD\longmapsto qD+2qT=q(D+2T).
\]
Thus the relations between the $\alpha$- and
$\beta$-denominators are preserved. The denominator of the first edge in \eqref{eq:path} is
\[
A'-x+z_1=A'+f(x)-x;
\]
this corresponds to the replaced old expansion edge, except that its
denominator has increased by $2T$. The subsequent denominators are
successively
\[
A',A'+f(z_1)-z_1,A',A'+f(z_2)-z_2,\cdots,A',A'+f(z_\ell)-z_\ell.
\]
Thus the new denominator multiset consists of the original
multiset with $n$ replaced by $n'$, together with $T$
additional copies of $A'$ and one additional copy of
$A'+s_{(z+r)\bmod L}$ for each
$z\in\{p+1,\ldots,p+T\}$. Since \(T\) is a multiple of \(L\), as \(z\) runs through $[p+1,p+T]$, each residue \((z+r)\bmod L\) occurs exactly \(T/L\) times; hence each step size \(s_k\) appears the same number of times. The increases in multiplicity are listed in
Table~\ref{tab:newmult}. Adding these increases to the
original multiplicities gives exactly the multiplicities
in Table~\ref{tab:denoms} with $n$ replaced by $n+2T$. The
number of $\beta$-denominators is unchanged, so the reciprocal sum of
the new cycle is still $1$.\par
\begin{table}[ht]
\centering
\caption{Additional multiplicities of $\alpha$-denominators}\label{tab:newmult}
\begin{tabular}{c|c}
\toprule
Type & Additional multiplicity in the order of Table~\ref{tab:denoms} \\
\midrule
0  & $(12,4,8)$ \\
2  & $(24,16,8)$ \\
5  & $(18,6,9,3)$ \\
11 & $(36,6,6,24)$ \\
1  & $(6,4,2)$ \\
7  & $(6,2,2,2)$ \\
3  & $(6,2,2,2)$ \\
4  & $(6,1,2,3)$ \\
10 & $(12,8,4)$ \\
\bottomrule
\end{tabular}
\end{table}
Now, we check that the new cyclic arrangement still satisfies
(H1)--(H4). Under the map \(E\), the old numbers retain their original
region indices. All new numbers lie in the new middle region, with
index $1$. For every $x\in J$, the first edge of the inserted path is
precisely the new expansion edge
\[
\{A'-x,f(x)\}.
\]
For the remaining new edges in the inserted path, the sum of the
endpoint region indices is $2$. The region relations of unchanged
edges are also preserved. Moreover,
\[
1\le A'-x\le n',\qquad p<f(x)<A'-p,\qquad A'>2p.
\]
Hence (H1)--(H3) continue to hold. The two splicing edges are not
replaced, and become
\[
(a+2T,b),\qquad(c+T,d+T).
\]
Both have denominator $(D+2T)$, and
\[
b+(d+T)=\frac{D}{2}+T=\frac{(D+2T)}{2}.
\]
Hence (H4) continues to hold, and Proposition~\ref{lem:splice} can be applied to obtain
the linear arrangement.

\section{Base cases and coverage}\label{sec:base}

\subsection*{Base cases.}
Appendix~\ref{app:start} gives all $28$ base cases. Their parameters are
shown in Table~\ref{tab:startparams} below; in the ``reversed'' column, ``yes'' means that
the entire cycle in the appendix is first reversed and then all
directed edges are read; ``no'' means reading in the original
direction.
\begin{table}[ht]
\centering
\caption{Parameters of base cases}\label{tab:startparams}
\small
\begin{tabular}{c|c|c|c|c|c|c}
\toprule
Type & $n$ & $p$ & $r$ & Reversed & $(a,b)$ & $(c,d)$ \\
\midrule
0  & 24 & 6  & 1 & no  & $(19,3)$  & $(14,8)$  \\
0  & 30 & 6  & 2 & no  & $(23,5)$  & $(19,9)$  \\
0  & 36 & 6  & 2 & no  & $(28,6)$  & $(23,11)$ \\
0  & 42 & 6  & 2 & yes & $(37,3)$  & $(23,17)$ \\
2  & 56 & 14 & 1 & no  & $(49,3)$  & $(29,23)$ \\
2  & 62 & 17 & 0 & no  & $(54,4)$  & $(33,25)$ \\
2  & 68 & 14 & 1 & no  & $(52,12)$ & $(44,20)$ \\
2  & 74 & 14 & 1 & no  & $(62,8)$  & $(43,27)$ \\
2  & 80 & 14 & 1 & no  & $(74,2)$  & $(40,36)$ \\
2  & 86 & 14 & 1 & no  & $(85,11)$ & $(59,37)$ \\
2  & 92 & 14 & 1 & no  & $(83,5)$  & $(49,39)$ \\
2  & 98 & 14 & 1 & no  & $(94,14)$ & $(68,40)$ \\
5  & 53 & 10 & 4 & no  & $(39,9)$  & $(33,15)$ \\
5  & 65 & 11 & 3 & no  & $(63,5)$  & $(39,29)$ \\
5  & 77 & 14 & 1 & no  & $(75,5)$  & $(45,35)$ \\
11 & 47 & 8  & 4 & yes & $(43,7)$  & $(32,18)$ \\
11 & 59 & 8  & 4 & no  & $(57,5)$  & $(36,26)$ \\
11 & 71 & 8  & 0 & no  & $(71,1)$  & $(37,35)$ \\
11 & 83 & 8  & 0 & yes & $(83,3)$  & $(46,40)$ \\
11 & 95 & 8  & 0 & no  & $(95,1)$  & $(49,47)$ \\
11 & 107& 8  & 0 & no  & $(105,5)$ & $(60,50)$ \\
1  & 25 & 3  & 0 & yes & $(25,1)$  & $(14,12)$ \\
7  & 31 & 4  & 0 & yes & $(26,4)$  & $(19,11)$ \\
3  & 27 & 3  & 0 & yes & $(25,3)$  & $(17,11)$ \\
3  & 33 & 3  & 1 & no  & $(29,3)$  & $(19,13)$ \\
4  & 28 & 3  & 3 & yes & $(24,2)$  & $(15,11)$ \\
10 & 34 & 5  & 2 & no  & $(33,3)$  & $(21,15)$ \\
10 & 46 & 5  & 2 & no  & $(44,4)$  & $(28,20)$ \\
\bottomrule
\end{tabular}
\end{table}
By direct finite computation, we obtain the following
proposition:
\begin{proposition}\label{prop:start}
Every cycle listed in Appendix~\ref{app:start} satisfies the denominator
multiplicities of Table~\ref{tab:denoms} and (H1)--(H4).
\end{proposition}
\medskip
\subsection*{Coverage of all cases.}
\begin{proposition}\label{lem:cover}
The induction chains in Table~\ref{tab:chains} together with the finite exceptions
in the last column, i.e., smaller exceptional n not covered by the induction chains, cover all $n>7$.
\end{proposition}
\begin{table}[ht]
\centering
\caption{Induction chains and finite exceptions}\label{tab:chains}
\small
\begin{tabular}{c|c|c|l}
\toprule
Type & base case of $n$ & Increment per step & finite exceptions  \\
\midrule
0  & $(24,30,36,42)$                          & 24 & $(12,18)$ \\
2  & $(56,62,68,74,80,86,92,98)$              & 48 & $(8,14,20,26,32,38,44,50)$ \\
5  & $(53,65,77)$                             & 36 & $(17,29,41)$ \\
11 & $(47,59,71,83,95,107)$                   & 72 & $(11,23,35)$ \\
1  & $(25)$                                   & 12 & $(13)$ \\
7  & $(31)$                                   & 12 & $(19)$ \\
3  & $(27,33)$                                & 12 & $(9,15,21)$ \\
4  & $(28)$                                   & 12 & $(16)$ \\
10 & $(34,46)$                                & 24 & $(10,22)$ \\
\bottomrule
\end{tabular}
\end{table}

\begin{proof}
Consider type $2$, which increments by 48 at each step. The numbers $n\equiv2\pmod6$
split into eight residue classes modulo $48$, and the starting values
\[
(56,62,68,74,80,86,92,98)
\]
meet each of these eight classes exactly once. Thus, for every sufficiently
large $n\equiv2\pmod6$, there is a unique starting value from which $n$ is
reached by repeatedly adding $48$; the remaining small values are
\[
(8,14,20,26,32,38,44,50).
\]
The other eight types are handled in the same way, with step sizes
$24,36,72,12,12,12,12,24$, respectively.
\end{proof}
Combining the base cases of Proposition~\ref{prop:start}, the inductive step of Proposition~\ref{lem:main-recurse}, and the coverage established in Proposition~\ref{lem:cover}, together with the explicit small exceptions listed in Appendix~\ref{app:exceptions}, Theorem~\ref{thm:main} holds.

\section*{Note on AI use}
Most of the construction and the recursive step in this paper were
proposed and refined through our interaction with GPT-5.6 Sol. In the course of exploration, we suggested constructing separately for different residue classes modulo $6$, and further unified the different residue classes by the recursive step.
We then conducted AI-assisted literature searches and found that Theorem~\ref{thm:main} was previously unproven. In the course of this search, we also collected numerous conjectures and results concerning permutations.\par
We independently checked the constructions and proofs for all cases
item by item, including the identity of the reciprocal sum, the
preservation of the conditions required by the inductive step, and the
coverage of all cases.
\appendix
\setcounter{tocdepth}{0}
\section{\texorpdfstring{Tables of $W_x$ (organized by type and $(p,r)$)}{Tables of Wx (organized by type and (p,r))}}\label{app:wx}
Each table fixes a type and $(p,r)$. The second column is $W_x$; the
third column is the next term, equal to the first term of the second
column plus $T$. Together, the second columns of each table contain
$p+1,\ldots,p+T$ exactly once each.
\begin{typecase}
Type 0, $(p=6,r=1)$.
\end{typecase}
\begin{center}
\begin{tabular}{c|c|c}
$x$ & $W_x$ & Next term \\
\hline
3 & $(8,10,15)$          & 20 \\
4 & $(9,14,16)$          & 21 \\
5 & $(7,12,17)$          & 19 \\
6 & $(11,13,18)$         & 23 \\
\end{tabular}
\end{center}
\begin{typecase}
Type 0, $(p=6,r=2)$.
\end{typecase}
\begin{center}
\begin{tabular}{c|c|c}
$x$ & $W_x$ & Next term \\
\hline
2 & $(7,9,14)$           & 19 \\
3 & $(8,13,15)$          & 20 \\
5 & $(10,12,17)$         & 22 \\
6 & $(11,16,18)$         & 23 \\
\end{tabular}
\end{center}
\begin{typecase}
Type 2, $(p=14,r=1)$.
\end{typecase}
\begin{center}
\begin{tabular}{c|c|c}
$x$ & $W_x$ & Next term \\
\hline
2  & $(16,21,26)$                & 40 \\
5  & $(19,24,29)$                & 43 \\
8  & $(22,27,32)$                & 46 \\
10 & $(15,20,34)$                & 39 \\
11 & $(25,30,35)$                & 49 \\
12 & $(17,31,36)$                & 41 \\
13 & $(18,23,37)$                & 42 \\
14 & $(28,33,38)$                & 52 \\
\end{tabular}
\end{center}
\begin{typecase}
Type 2, $(p=17,r=0)$.
\end{typecase}
\begin{center}
\begin{tabular}{c|c|c}
$x$ & $W_x$ & Next term \\
\hline
6  & $(20,25,30)$                & 44 \\
9  & $(23,28,33)$                & 47 \\
12 & $(26,31,36)$                & 50 \\
13 & $(18,32,37)$                & 42 \\
14 & $(19,24,38)$                & 43 \\
15 & $(29,34,39)$                & 53 \\
16 & $(21,35,40)$                & 45 \\
17 & $(22,27,41)$                & 46 \\
\end{tabular}
\end{center}
\begin{typecase}
Type 5, $(p=10,r=4)$.
\end{typecase}
\begin{center}
\begin{tabular}{c|c|c}
$x$ & $W_x$ & Next term \\
\hline
0 & $(11,18)$                                   & 29 \\
5 & $(12,23)$                                   & 30 \\
6 & $(17,24)$                                   & 35 \\
7 & $(14,13,20,19,26,25)$                       & 32 \\
9 & $(16,15,22,21,28,27)$                       & 34 \\
\end{tabular}
\end{center}
\begin{typecase}
Type 5, $(p=11,r=3)$.
\end{typecase}
\begin{center}
\begin{tabular}{c|c|c}
$x$ & $W_x$ & Next term \\
\hline
1  & $(12,19)$                                   & 30 \\
6  & $(13,24)$                                   & 31 \\
7  & $(18,25)$                                   & 36 \\
8  & $(15,14,21,20,27,26)$                       & 33 \\
10 & $(17,16,23,22,29,28)$                       & 35 \\
\end{tabular}
\end{center}
\begin{typecase}
Type 5, $(p=14,r=1)$.
\end{typecase}
\begin{center}
\begin{tabular}{c|c|c}
$x$ & $W_x$ & Next term \\
\hline
8  & $(15,26)$                                   & 33 \\
9  & $(20,27)$                                   & 38 \\
10 & $(17,16,23,22,29,28)$                       & 35 \\
12 & $(19,18,25,24,31,30)$                       & 37 \\
14 & $(21,32)$                                   & 39 \\
\end{tabular}
\end{center}
\begin{typecase}
Type 11, $(p=8,r=4)$.
\end{typecase}
\begin{center}
\begin{tabular}{c|c|c}
$x$ & $W_x$ & Next term \\
\hline
3 & $(10,15,22,27,34,39)$                            & 46 \\
4 & $(9,16,21,28,33,40)$                             & 45 \\
5 & $(12,19,26,23,30,37,44,41)$                      & 48 \\
6 & $(13,20,17,24,31,38,35,42)$                      & 49 \\
7 & $(14,11,18,25,32,29,36,43)$                      & 50 \\
\end{tabular}
\end{center}
\begin{typecase}
Type 11, $(p=8,r=0)$.
\end{typecase}
\begin{center}
\begin{tabular}{c|c|c}
$x$ & $W_x$ & Next term \\
\hline
3 & $(10,17,24,21,28,35,42,39)$                      & 46 \\
4 & $(11,18,15,22,29,36,33,40)$                      & 47 \\
5 & $(12,9,16,23,30,27,34,41)$                       & 48 \\
7 & $(14,19,26,31,38,43)$                            & 50 \\
8 & $(13,20,25,32,37,44)$                            & 49 \\
\end{tabular}
\end{center}
\begin{typecase}
Type 1, $(p=3,r=0)$.
\end{typecase}
\begin{center}
\begin{tabular}{c|c|c}
$x$ & $W_x$ & Next term \\
\hline
1 & $(5,9,7)$            & 11 \\
2 & $(6,4,8)$            & 12 \\
\end{tabular}
\end{center}
\begin{typecase}
Type 7, $(p=4,r=0)$.
\end{typecase}
\begin{center}
\begin{tabular}{c|c|c}
$x$ & $W_x$ & Next term \\
\hline
4 & $(7,6,5,8,9,10)$     & 13 \\
\end{tabular}
\end{center}
\begin{typecase}
Type 3, $(p=3,r=0)$.
\end{typecase}
\begin{center}
\begin{tabular}{c|c|c}
$x$ & $W_x$ & Next term \\
\hline
3 & $(4,7,6,5,8,9)$      & 10 \\
\end{tabular}
\end{center}
\begin{typecase}
Type 3, $(p=3,r=1)$.
\end{typecase}
\begin{center}
\begin{tabular}{c|c|c}
$x$ & $W_x$ & Next term \\
\hline
3 & $(6,5,4,7,8,9)$      & 12 \\
\end{tabular}
\end{center}
\begin{typecase}
Type 4, $(p=3,r=3)$.
\end{typecase}
\begin{center}
\begin{tabular}{c|c|c}
$x$ & $W_x$ & Next term \\
\hline
$-2$ & $(4)$   & 10 \\
$0$  & $(6)$   & 12 \\
$1$  & $(5,9,7)$ & 11 \\
$2$  & $(8)$   & 14 \\
\end{tabular}
\end{center}
\begin{typecase}
Type 10, $(p=5,r=2)$.
\end{typecase}
\begin{center}
\begin{tabular}{c|c|c}
$x$ & $W_x$ & Next term \\
\hline
0 & $(8,10,12)$          & 20 \\
3 & $(11,13,15)$         & 23 \\
4 & $(6,14,16)$          & 18 \\
5 & $(7,9,17)$           & 19 \\
\end{tabular}
\end{center}
\section{Induction base cases of cyclic arrangements}\label{app:start}
Each parenthesis below is interpreted as $(\tau(1),\ldots,\tau(n))$,
with the last entry adjacent to the first.
\begin{typecase}
Type 0, $(n=24)$.
\end{typecase}
\seq{(1, 21, 6, 16, 11, 22, 5, 17, 7, 15, 12, 10, 2, 20, 13, 14, 8, 19, 3, 24, 9, 18, 4, 23)}
\begin{typecase}
Type 0, $(n=30)$.
\end{typecase}
\seq{(1, 27, 15, 18, 10, 23, 5, 28, 14, 19, 9, 21, 7, 26, 2, 13, 20, 8, 25, 3, 30, 12, 16, 17, 11, 22, 6, 24, 4, 29)}
\begin{typecase}
Type 0, $(n=36)$.
\end{typecase}
\seq{(1, 33, 18, 16, 23, 11, 28, 6, 12, 22, 17, 19, 20, 14, 25, 9, 27, 7, 32, 2, 34, 5, 29, 10, 24, 15, 36, 3, 31, 8, 26, 13, 21, 30, 4, 35)}
\begin{typecase}
Type 0, $(n=42)$.
\end{typecase}
\seq{(1, 39, 6, 34, 11, 29, 16, 5, 35, 10, 30, 12, 28, 17, 23, 22, 18, 42, 3, 37, 8, 32, 13, 27, 15, 25, 20, 40, 2, 38, 7, 33, 9, 31, 14, 26, 19, 21, 24, 36, 4, 41)}
\begin{typecase}
Type 2, $(n=56)$.
\end{typecase}
\seq{(1, 51, 6, 46, 11, 41, 25, 27, 30, 22, 44, 8, 49, 3, 54, 24, 33, 19, 47, 5, 52, 14, 38, 28, 29, 23, 55, 2, 50, 16, 36, 21, 31, 26, 7, 45, 12, 40, 17, 35, 43, 9, 48, 4, 53, 13, 39, 18, 34, 32, 20, 37, 15, 42, 10, 56)}
\begin{typecase}
Type 2, $(n=62)$.
\end{typecase}
\seq{(1, 57, 6, 52, 20, 38, 34, 24, 39, 19, 44, 14, 58, 5, 53, 10, 48, 15, 43, 29, 7, 51, 12, 60, 3, 55, 8, 50, 13, 45, 18, 40, 32, 26, 46, 17, 41, 22, 36, 27, 31, 56, 2, 61, 11, 47, 16, 42, 21, 37, 35, 23, 49, 9, 54, 4, 59, 28, 30, 33, 25, 62)}
\begin{typecase}
Type 2, $(n=68)$.
\end{typecase}
\seq{(1, 63, 6, 58, 38, 26, 43, 21, 48, 16, 62, 2, 67, 11, 53, 25, 39, 30, 66, 3, 61, 8, 56, 22, 42, 27, 37, 41, 23, 46, 18, 51, 13, 65, 4, 60, 9, 55, 14, 50, 28, 36, 33, 31, 47, 17, 52, 12, 57, 7, 32, 64, 5, 59, 19, 45, 24, 40, 29, 35, 34, 44, 20, 49, 15, 54, 10, 68)}
\begin{typecase}
Type 2, $(n=74)$.
\end{typecase}
\seq{(1, 69, 6, 64, 11, 59, 25, 45, 30, 40, 35, 7, 63, 42, 33, 37, 38, 32, 73, 2, 68, 16, 54, 21, 49, 26, 44, 61, 9, 66, 4, 71, 13, 57, 18, 52, 23, 47, 28, 56, 14, 70, 5, 65, 19, 51, 24, 46, 29, 41, 43, 27, 48, 22, 62, 8, 67, 3, 72, 12, 58, 17, 53, 31, 39, 36, 34, 50, 20, 55, 15, 60, 10, 74)}
\begin{typecase}
Type 2, $(n=80)$.
\end{typecase}
\seq{(1, 75, 6, 70, 11, 65, 25, 51, 30, 46, 44, 32, 49, 27, 54, 22, 68, 8, 73, 3, 78, 12, 64, 17, 59, 31, 50, 26, 55, 21, 60, 16, 74, 2, 79, 35, 41, 40, 36, 45, 69, 7, 38, 43, 33, 48, 28, 62, 14, 76, 5, 71, 19, 57, 24, 52, 29, 47, 67, 9, 72, 4, 77, 13, 63, 18, 58, 23, 53, 37, 39, 42, 34, 56, 20, 61, 15, 66, 10, 80)}
\begin{typecase}
Type 2, $(n=86)$.
\end{typecase}
\seq{(1, 81, 6, 76, 47, 35, 52, 30, 57, 39, 43, 44, 38, 49, 33, 54, 28, 68, 14, 82, 5, 77, 19, 63, 24, 58, 29, 53, 34, 48, 75, 7, 41, 46, 36, 51, 31, 65, 17, 70, 12, 84, 3, 79, 8, 74, 22, 60, 27, 55, 32, 50, 73, 9, 78, 4, 83, 13, 69, 18, 64, 23, 59, 37, 45, 42, 40, 56, 26, 61, 21, 66, 16, 80, 2, 85, 11, 71, 25, 62, 20, 67, 15, 72, 10, 86)}
\begin{typecase}
Type 2, $(n=92)$.
\end{typecase}
\seq{(1, 87, 6, 82, 11, 77, 25, 63, 30, 58, 35, 53, 79, 9, 84, 4, 89, 13, 75, 18, 70, 23, 65, 37, 51, 42, 90, 3, 85, 8, 80, 22, 66, 27, 61, 32, 56, 46, 47, 41, 91, 2, 86, 16, 72, 21, 67, 26, 62, 40, 48, 45, 43, 59, 29, 64, 24, 69, 19, 83, 5, 88, 14, 74, 28, 60, 33, 55, 38, 50, 52, 36, 57, 31, 71, 17, 76, 12, 81, 7, 44, 49, 39, 54, 34, 68, 20, 73, 15, 78, 10, 92)}
\begin{typecase}
Type 2, $(n=98)$.
\end{typecase}
\seq{(1, 93, 6, 88, 11, 83, 25, 69, 30, 64, 35, 59, 49, 50, 44, 97, 2, 92, 16, 78, 21, 73, 26, 68, 40, 54, 45, 96, 3, 91, 8, 86, 22, 72, 27, 67, 32, 62, 46, 48, 51, 43, 65, 29, 70, 24, 75, 19, 89, 5, 94, 14, 80, 28, 66, 33, 61, 38, 56, 85, 9, 90, 4, 95, 13, 81, 18, 76, 23, 71, 37, 57, 42, 52, 47, 7, 87, 12, 82, 17, 77, 31, 63, 36, 58, 41, 53, 55, 39, 60, 34, 74, 20, 79, 15, 84, 10, 98)}
\begin{typecase}
Type 5, $(n=53)$.
\end{typecase}
\seq{(1, 48, 8, 41, 7, 42, 14, 35, 13, 36, 24, 25, 31, 18, 38, 11, 49, 23, 26, 30, 19, 37, 12, 44, 5, 51, 21, 28, 32, 17, 43, 6, 50, 10, 39, 9, 40, 16, 33, 15, 34, 22, 27, 29, 20, 52, 4, 45, 53, 3, 46, 2, 47)}
\begin{typecase}
Type 5, $(n=65)$.
\end{typecase}
\seq{(1, 60, 12, 49, 19, 42, 30, 31, 37, 24, 36, 25, 43, 18, 54, 7, 61, 11, 50, 10, 51, 17, 44, 16, 45, 23, 38, 22, 39, 29, 32, 40, 21, 47, 14, 46, 15, 53, 8, 52, 9, 63, 5, 56, 34, 27, 41, 20, 48, 13, 55, 6, 62, 28, 33, 35, 26, 64, 4, 57, 65, 3, 58, 2, 59)}
\begin{typecase}
Type 5, $(n=77)$.
\end{typecase}
\seq{(1, 72, 36, 37, 43, 30, 42, 31, 49, 24, 48, 25, 55, 18, 54, 19, 61, 12, 60, 13, 67, 6, 74, 34, 39, 41, 32, 52, 21, 59, 14, 66, 7, 73, 11, 62, 10, 63, 17, 56, 16, 57, 23, 50, 22, 51, 29, 44, 28, 45, 35, 38, 46, 27, 53, 20, 64, 9, 75, 5, 68, 40, 33, 47, 26, 58, 15, 65, 8, 76, 4, 69, 77, 3, 70, 2, 71)}
\begin{typecase}
Type 11, $(n=47)$.
\end{typecase}
\seq{(1, 42, 8, 35, 25, 18, 32, 11, 29, 14, 36, 7, 43, 5, 38, 12, 31, 44, 6, 37, 13, 30, 20, 23, 27, 16, 34, 9, 39, 4, 46, 2, 41, 45, 15, 28, 22, 21, 19, 24, 26, 17, 33, 10, 40, 3, 47)}
\begin{typecase}
Type 11, $(n=59)$.
\end{typecase}
\seq{(1, 54, 8, 47, 31, 24, 38, 17, 35, 20, 42, 13, 49, 6, 56, 37, 18, 44, 11, 41, 14, 48, 7, 55, 23, 32, 30, 25, 27, 28, 34, 21, 39, 16, 46, 9, 51, 4, 58, 2, 53, 57, 5, 50, 12, 43, 19, 36, 26, 29, 33, 22, 40, 15, 45, 10, 52, 3, 59)}
\begin{typecase}
Type 11, $(n=71)$.
\end{typecase}
\seq{(1, 66, 8, 59, 13, 54, 20, 47, 27, 40, 34, 33, 41, 26, 46, 21, 53, 14, 60, 7, 67, 29, 38, 36, 31, 43, 24, 48, 19, 45, 22, 52, 15, 49, 18, 56, 11, 63, 4, 70, 2, 65, 69, 5, 62, 12, 55, 9, 58, 16, 51, 23, 44, 30, 37, 35, 32, 42, 25, 39, 28, 68, 6, 61, 50, 17, 57, 10, 64, 3, 71)}
\begin{typecase}
Type 11, $(n=83)$.
\end{typecase}
\seq{(1, 78, 8, 71, 13, 66, 20, 59, 27, 52, 32, 47, 39, 40, 46, 33, 43, 36, 50, 29, 57, 22, 64, 15, 61, 18, 68, 11, 75, 4, 82, 2, 77, 81, 5, 74, 12, 67, 9, 70, 16, 63, 23, 56, 73, 6, 80, 34, 45, 41, 38, 48, 31, 53, 26, 60, 19, 65, 14, 72, 7, 79, 35, 44, 42, 37, 49, 30, 54, 25, 51, 28, 58, 21, 55, 24, 62, 17, 69, 10, 76, 3, 83)}
\begin{typecase}
Type 11, $(n=95)$.
\end{typecase}
\seq{(1, 90, 8, 83, 13, 78, 20, 71, 25, 66, 32, 59, 39, 52, 46, 45, 51, 40, 58, 33, 65, 26, 72, 19, 77, 14, 84, 7, 91, 41, 50, 48, 43, 53, 38, 60, 31, 57, 34, 64, 27, 61, 30, 68, 23, 75, 16, 82, 9, 79, 12, 86, 5, 93, 89, 2, 94, 4, 87, 11, 80, 18, 73, 15, 76, 22, 69, 29, 62, 36, 55, 92, 6, 85, 3, 88, 10, 81, 17, 74, 24, 67, 21, 70, 28, 63, 35, 56, 42, 49, 47, 44, 54, 37, 95)}
\begin{typecase}
Type 11, $(n=107)$.
\end{typecase}
\seq{(1, 102, 8, 95, 13, 90, 20, 83, 27, 76, 34, 69, 31, 72, 38, 65, 35, 68, 97, 6, 104, 46, 57, 43, 60, 50, 53, 55, 48, 62, 41, 67, 36, 74, 29, 81, 22, 88, 15, 85, 18, 92, 11, 99, 4, 106, 2, 101, 105, 5, 98, 12, 91, 9, 94, 16, 87, 23, 80, 30, 73, 37, 66, 42, 61, 49, 54, 56, 47, 103, 7, 96, 14, 89, 19, 84, 26, 77, 33, 70, 40, 63, 45, 58, 52, 51, 59, 44, 64, 39, 71, 32, 78, 25, 75, 28, 82, 21, 79, 24, 86, 17, 93, 10, 100, 3, 107)}
\begin{typecase}
Type 1, $(n=25)$.
\end{typecase}
\seq{(1, 25, 5, 21, 9, 15, 11, 19, 7, 17, 13, 2, 24, 6, 20, 4, 22, 8, 18, 12, 14, 16, 10, 3, 23)}
\begin{typecase}
Type 7, $(n=31)$.
\end{typecase}
\seq{(1, 28, 14, 15, 13, 16, 12, 17, 31, 27, 2, 30, 18, 11, 19, 10, 20, 9, 21, 8, 24, 5, 23, 6, 22, 7, 25, 4, 26, 3, 29)}
\begin{typecase}
Type 3, $(n=27)$.
\end{typecase}
\seq{(1, 24, 2, 23, 27, 12, 13, 26, 16, 9, 15, 10, 14, 11, 17, 8, 20, 5, 19, 6, 18, 7, 21, 4, 22, 3, 25)}
\begin{typecase}
Type 3, $(n=33)$.
\end{typecase}
\seq{(1, 30, 32, 2, 29, 3, 28, 6, 25, 5, 26, 4, 27, 7, 24, 10, 21, 9, 22, 8, 23, 11, 20, 12, 19, 13, 18, 33, 15, 16, 14, 17, 31)}
\begin{typecase}
Type 4, $(n=28)$.
\end{typecase}
\seq{(1, 25, 5, 21, 11, 15, 9, 17, 7, 19, 13, 26, 6, 20, 12, 27, 3, 23, 16, 10, 22, 4, 28, 2, 24, 8, 18, 14)}
\begin{typecase}
Type 10, $(n=34)$.
\end{typecase}
\seq{(1, 33, 3, 31, 11, 23, 13, 21, 15, 19, 17, 25, 9, 27, 7, 29, 5, 12, 22, 14, 28, 6, 30, 4, 32, 2, 34, 8, 26, 10, 24, 18, 16, 20)}
\begin{typecase}
Type 10, $(n=46)$.
\end{typecase}
\seq{(1, 45, 3, 43, 11, 35, 13, 33, 15, 31, 23, 25, 21, 27, 19, 29, 17, 37, 9, 39, 7, 41, 5, 18, 28, 20, 34, 12, 36, 10, 38, 8, 46, 2, 44, 4, 42, 6, 40, 14, 32, 16, 30, 24, 22, 26)}
\section{\texorpdfstring{Small exceptional $n$ not covered by the induction chains}{Small exceptional n not covered by the induction chains}}\label{app:exceptions}
For each $n$ below, we list the cyclic arrangement and the linear
arrangement.
\begin{typecase}
$(n=7)$.
\end{typecase}
Linear arrangement:
\seq{(4, 5, 7, 2, 1, 3, 6)}
\begin{typecase}
$(n=8)$.
\end{typecase}
Cyclic arrangement:
\seq{(1, 3, 5, 4, 6, 2, 7, 8)}
Linear arrangement:
\seq{(3, 4, 8, 6, 1, 2, 5, 7)}
\begin{typecase}
$(n=9)$.
\end{typecase}
Cyclic arrangement:
\seq{(1, 8, 7, 2, 6, 9, 3, 5, 4)}
Linear arrangement:
\seq{(7, 8, 1, 5, 4, 2, 3, 6, 9)}
\begin{typecase}
$(n=10)$.
\end{typecase}
Cyclic arrangement:
\seq{(1, 5, 3, 9, 6, 4, 8, 10, 2, 7)}
Linear arrangement:
\seq{(6, 9, 1, 3, 7, 5, 10, 2, 4, 8)}
\begin{typecase}
$(n=11)$.
\end{typecase}
Cyclic arrangement:
\seq{(1, 5, 3, 7, 9, 6, 10, 2, 8, 4, 11)}
Linear arrangement:
\seq{(7, 2, 4, 10, 11, 3, 6, 8, 1, 5, 9)}
\begin{typecase}
$(n=12)$.
\end{typecase}
Cyclic arrangement:
\seq{(1, 7, 5, 10, 2, 6, 4, 8, 12, 3, 9, 11)}
Linear arrangement:
\seq{(2, 3, 12, 6, 9, 1, 4, 11, 7, 8, 10, 5)}
\begin{typecase}
$(n=13)$.
\end{typecase}
Cyclic arrangement:
\seq{(1, 11, 7, 2, 12, 6, 8, 10, 4, 3, 9, 5, 13)}
Linear arrangement:
\seq{(10, 8, 4, 3, 9, 5, 13, 1, 6, 12, 2, 7, 11)}
\begin{typecase}
$(n=14)$.
\end{typecase}
Cyclic arrangement:
\seq{(1, 11, 4, 12, 8, 10, 5, 13, 3, 6, 14, 2, 7, 9)}
Linear arrangement:
\seq{(4, 1, 14, 10, 5, 7, 3, 12, 8, 2, 13, 11, 9, 6)}
\begin{typecase}
$(n=15)$.
\end{typecase}
Cyclic arrangement:
\seq{(1, 14, 2, 13, 11, 3, 12, 4, 10, 5, 9, 15, 6, 8, 7)}
Linear arrangement:
\seq{(9, 15, 6, 8, 7, 1, 14, 10, 5, 2, 13, 3, 11, 4, 12)}
\begin{typecase}
$(n=16)$.
\end{typecase}
Cyclic arrangement:
\seq{(1, 11, 13, 3, 9, 12, 2, 14, 7, 5, 16, 8, 4, 10, 6, 15)}
Linear arrangement:
\seq{(3, 9, 7, 5, 11, 1, 13, 15, 6, 2, 12, 16, 8, 4, 10, 14)}
\begin{typecase}
$(n=17)$.
\end{typecase}
Cyclic arrangement:
\seq{(1, 17, 9, 4, 16, 2, 11, 13, 7, 6, 14, 10, 3, 15, 5, 8, 12)}
Linear arrangement:
\seq{(1, 12, 14, 6, 7, 13, 11, 2, 4, 16, 10, 3, 17, 9, 15, 5, 8)}
\begin{typecase}
$(n=18)$.
\end{typecase}
Cyclic arrangement:
\seq{(1, 15, 9, 12, 4, 17, 7, 14, 2, 16, 5, 11, 10, 6, 18, 3, 13, 8)}
Linear arrangement:
\seq{(2, 16, 5, 11, 10, 6, 15, 17, 4, 12, 9, 7, 1, 8, 13, 3, 18, 14)}
\begin{typecase}
$(n=19)$.
\end{typecase}
Cyclic arrangement:
\seq{(1, 17, 5, 19, 3, 15, 7, 11, 13, 9, 2, 16, 8, 14, 4, 18, 6, 12, 10)}
Linear arrangement:
\seq{(12, 10, 1, 17, 5, 19, 3, 6, 18, 4, 14, 8, 16, 2, 9, 13, 11, 7, 15)}
\begin{typecase}
$(n=20)$.
\end{typecase}
Cyclic arrangement:
\seq{(1, 15, 9, 7, 14, 2, 19, 11, 5, 16, 8, 13, 17, 4, 12, 3, 18, 6, 10, 20)}
Linear arrangement:
\seq{(5, 16, 14, 18, 3, 13, 8, 7, 9, 12, 4, 17, 15, 1, 20, 10, 6, 2, 19, 11)}
\begin{typecase}
$(n=21)$.
\end{typecase}
Cyclic arrangement:
\seq{(1, 20, 2, 19, 14, 6, 15, 5, 16, 17, 3, 18, 4, 7, 13, 8, 12, 10, 11, 9, 21)}
Linear arrangement:
\seq{(10, 11, 19, 2, 18, 4, 6, 15, 5, 16, 17, 3, 8, 14, 7, 13, 9, 12, 21, 1, 20)}
\begin{typecase}
$(n=22)$.
\end{typecase}
Cyclic arrangement:
\seq{(11, 7, 15, 3, 17, 5, 13,9, 21, 1, 19, 14, 4, 16, 6, 12, 18, 2, 20, 10, 8, 22)}
Linear arrangement:
\seq{(17, 3, 15, 7, 11, 22, 8, 10, 20, 2, 18, 12, 6, 5, 13,9, 21, 1, 19, 14, 4, 16)}
\begin{typecase}
$(n=23)$.
\end{typecase}
Cyclic arrangement:
\seq{(1, 23, 15, 4, 22, 2, 17, 9, 10, 16, 3, 21, 5, 14, 12, 7, 19, 20, 6, 13, 11, 8, 18)}
Linear arrangement:
\seq{(18, 1, 12, 7, 19, 20, 6, 13, 11, 8, 16, 3, 23, 15, 4, 22, 2, 17, 9, 10, 14, 5, 21)}
\begin{typecase}
$(n=26)$.
\end{typecase}
Cyclic arrangement:
\seq{(1, 21, 6, 16, 11, 25, 2, 20, 13, 23, 4, 18, 9, 24, 12, 10, 17, 5, 22, 14, 8, 19, 3, 15, 7, 26)}
Linear arrangement:
\seq{(11, 25, 2, 20, 16, 6, 21, 1, 26, 10, 17, 5, 22, 14, 8, 19, 3, 24, 12, 15, 7, 4, 18, 9, 13, 23)}
\begin{typecase}
$(n=29)$.
\end{typecase}
Cyclic arrangement:
\seq{(1, 27, 6, 18, 15, 9, 24, 4, 29, 13, 20, 8, 25, 3, 21, 12, 16, 17, 7, 26, 2, 22, 11, 10, 14, 28, 5, 19, 23)}
Linear arrangement:
\seq{(16, 9, 27, 5, 20, 12, 13, 19, 6, 18, 7, 25, 11, 14, 2, 23, 1, 24, 8, 17, 15, 10, 26, 22, 3, 29, 21, 4, 28)}
\begin{typecase}
$(n=32)$.
\end{typecase}
Cyclic arrangement:
\seq{(1, 31, 5, 22, 4, 32, 7, 29, 10, 17, 19, 20, 16, 23, 3, 24, 12, 15, 11, 28, 8, 18, 9, 27, 21, 6, 30, 2, 25, 14, 13, 26)}
Linear arrangement:
\seq{(21, 7, 26, 2, 31, 11, 17, 16, 12, 30, 3, 25, 8, 20, 22, 6, 27, 1, 32, 10, 18, 15, 13, 29, 4, 24, 9, 5, 23, 19, 14, 28)}
\begin{typecase}
$(n=35)$.
\end{typecase}
Cyclic arrangement:
\seq{(1, 30, 8, 23, 15, 16, 22, 9, 19, 12, 26, 5, 31, 7, 24, 33, 29, 2, 34, 4, 27, 11, 20, 18, 13, 25, 6, 32, 10, 21, 17, 14, 28, 3, 35)}
Linear arrangement:
\seq{(24, 7, 31, 5, 26, 12, 19, 9, 22, 16, 15, 23, 8, 30, 1, 35, 3, 28, 14, 17, 21, 10, 32, 25, 6, 13, 18, 20, 11, 27, 4, 34, 2, 29, 33)}
\begin{typecase}
$(n=38)$.
\end{typecase}
Cyclic arrangement:
\seq{(1, 33, 6, 28, 20, 14, 34, 5, 29, 19, 15, 36, 3, 31, 8, 26, 13, 21, 18, 16, 35, 4, 30, 9, 25, 23, 11, 37, 2, 32, 7, 17, 22, 12, 27, 24, 10, 38)}
Linear arrangement:
\seq{(30, 4, 35, 13, 21, 18, 16, 23, 11, 37, 2, 32, 7, 27, 12, 36, 3, 31, 17, 22, 26, 8, 9, 25, 14, 20, 28, 6, 33, 1, 38, 10, 24, 15, 19, 29, 5, 34)}
\begin{typecase}
$(n=41)$.
\end{typecase}
Cyclic arrangement:
\seq{(1, 35, 2, 34, 3, 41, 33, 4, 40, 14, 23, 21, 16, 20, 17, 27, 10, 38, 6, 31, 13, 24, 12, 25, 19, 18, 30, 7, 37, 11, 26, 28, 9, 39, 5, 32, 22, 15, 29, 8, 36)}
Linear arrangement:
\seq{(30, 7, 37, 11, 26, 18, 19, 25, 12, 24, 13, 31, 6, 38, 10, 27, 17, 20, 16, 21, 23, 14, 34, 3, 41, 33, 4, 40, 32, 5, 39, 9, 28, 8, 29, 15, 22, 2, 35, 1, 36)}
\begin{typecase}
$(n=44)$.
\end{typecase}
Cyclic arrangement:
\seq{(1, 39, 6, 34, 20, 40, 5, 35, 10, 30, 15, 25, 29, 11, 43, 2, 38, 7, 33, 21, 19, 26, 14, 31, 9, 36, 4, 41, 13, 27, 18, 22, 23, 17, 37, 8, 32, 28, 12, 42, 3, 24, 16, 44)}
Linear arrangement:
\seq{(36, 4, 41, 13, 27, 18, 22, 32, 8, 37, 3, 42, 12, 28, 17, 23, 31, 9, 11, 29, 16, 24, 30, 15, 25, 20, 34, 6, 39, 1, 44, 10, 35, 5, 40, 14, 26, 19, 21, 33, 7, 38, 2, 43)}
\begin{typecase}
$(n=50)$.
\end{typecase}
Cyclic arrangement:
\seq{(1, 45, 6, 40, 11, 35, 16, 30, 21, 25, 44, 2, 49, 20, 26, 34, 17, 29, 22, 24, 27, 19, 41, 5, 46, 14, 32, 28, 18, 33, 13, 47, 4, 42, 9, 37, 23, 7, 39, 12, 48, 3, 43, 8, 38, 31, 15, 36, 10, 50)}
Linear arrangement:
\seq{(46, 14, 32, 19, 27, 33, 13, 47, 4, 42, 9, 37, 23, 28, 18, 5, 41, 10, 50, 1, 45, 6, 40, 20, 26, 25, 21, 30, 16, 35, 11, 49, 2, 44, 7, 39, 12, 34, 17, 29, 31, 15, 36, 24, 22, 38, 8, 43, 3, 48)}

\end{document}